\documentclass[11pt]{article}
\usepackage[top=2.54cm, bottom=2.54cm, left=2.5cm, right=2.5cm]{geometry}

\usepackage[T1]{fontenc}
\usepackage[utf8]{inputenc}
\usepackage{lmodern}
\usepackage{microtype}

\usepackage[tbtags]{mathtools}
\usepackage{amssymb, latexsym}
\usepackage{amsthm}
\usepackage{mathrsfs}
\usepackage{enumitem}
\usepackage{aliascnt}

\usepackage{xcolor}
\usepackage[numbers,sort&compress]{natbib}
\usepackage[colorlinks=true, linkcolor=blue, citecolor=blue, urlcolor=magenta]{hyperref}
\usepackage[nameinlink,capitalize]{cleveref}
\usepackage{lineno}

\usepackage{graphicx}
\usepackage{appendix}
\usepackage{fancyhdr}

\usepackage{float}
\usepackage{pict2e}
\usepackage{tikz}
\usepackage{tikz-3dplot}
\usepackage{tkz-graph}
\usetikzlibrary{arrows,shapes,chains}
\usepackage{verbatim}
\usepackage{scalerel}

\allowdisplaybreaks[4]
\numberwithin{equation}{section}

\theoremstyle{plain}
\newtheorem{theorem}{Theorem}[section]

\newaliascnt{lemma}{theorem}
\newtheorem{lemma}[lemma]{Lemma}
\aliascntresetthe{lemma}

\newaliascnt{corollary}{theorem}
\newtheorem{corollary}[corollary]{Corollary}
\aliascntresetthe{corollary}

\theoremstyle{remark}
\newaliascnt{remark}{theorem}

\aliascntresetthe{remark}

\theoremstyle{definition}
\newaliascnt{definition}{theorem}
\newtheorem{definition}[definition]{Definition}
\aliascntresetthe{definition}

\newaliascnt{conjecture}{theorem}
\newtheorem{conjecture}[conjecture]{Conjecture}
\aliascntresetthe{conjecture}

\newaliascnt{problem}{theorem}

\aliascntresetthe{problem}

\newaliascnt{claim}{theorem}

\aliascntresetthe{claim}

\newaliascnt{observation}{theorem}

\aliascntresetthe{observation}

\newaliascnt{fact}{theorem}

\aliascntresetthe{fact}

\newaliascnt{construction}{theorem}
\newtheorem{construction}[construction]{Construction}
\aliascntresetthe{construction}

\crefname{theorem}{Theorem}{Theorems}
\crefname{lemma}{Lemma}{Lemmas}
\crefname{corollary}{Corollary}{Corollaries}
\crefname{remark}{Remark}{Remarks}
\crefname{problem}{Problem}{Problems}
\crefname{conjecture}{Conjecture}{Conjectures}
\crefname{definition}{Definition}{Definitions}
\crefname{claim}{Claim}{Claims}
\crefname{observation}{Observation}{Observations}
\crefname{fact}{Fact}{Facts}
\crefname{construction}{Construction}{Constructions}

\usepackage{nomencl}
\makenomenclature

\newcommand{\eps}{\varepsilon}
\newcommand{\cP}{\mathcal P}

\newcommand{\F}{\mathcal F}

\newcommand{\V}{\mathcal V}

\newcommand{\cA}{\mathcal A}
\newcommand{\cB}{\mathcal B}
\newcommand{\cC}{\mathcal C}
\newcommand{\cH}{\mathcal H}
\newcommand{\bd}{\partial}
\newcommand{\HC}{\mathrm{HC}}
\newcommand{\HF}{\mathrm{HF}}
\DeclareMathOperator{\Bin}{Bin}
\DeclareMathOperator{\Beta}{Beta}

\title{Large tight components and Dirac thresholds for tight Hamilton cycles}

\author{Hao Lin \thanks{School of Mathematical Sciences, University of Science and Technology of China, Hefei, Anhui 230026, China. Email: \texttt{haolinz6@ustc.edu.cn}.} 
\and
Guanghui Wang\thanks{School of Mathematics, Shandong University, Jinan, 250100,  China. 
Research supported by National Key R\&D Program of China
(2020YFA0712400) and National Natural Science Foundation of China (12231018).
Email: \texttt{ghwang@sdu.edu.cn}.}
\and
Jiande Zhang\thanks{School of Mathematics, Shandong University,
Jinan 250100, China. Email: \texttt{202511912@mail.sdu.edu.cn}.}
\and
Wenling Zhou \thanks{School of Mathematics, Shandong University, Jinan, Shandong 250100, China. Research supported by the National Natural Science Foundation of China (No. 12401457),  Shandong Provincial Natural Science Foundation (No. ZR2024QA067), China Postdoctoral Science Foundation (No. 2024M761780), and Young Talent of Lifting engineering for Science and Technology in Shandong, China (No. SDAST2025QTA074). Email: \texttt{gracezhou@sdu.edu.cn}.
}
}

\date{}
\begin{document}
\maketitle

\begin{abstract}
For every odd integer $r\ge1$, let
$q_r=1-2^{-r}\binom r{(r-1)/2}$.
We show that, for every odd $r\ge3$ and sufficiently large $n$, every $n$-vertex $r$-uniform hypergraph with at least
$q_r\binom nr$ edges contains a tight component with
at least $(1/2-o(1))\binom nr$ edges.
This confirms a conjecture posed by Lang and Sanhueza-Matamala in 2024.
As an application, for every odd $r\ge1$ and every integer
$k\ge\lceil r(3-2q_r)\rceil$, we prove that the asymptotic minimum $(k-r)$-degree threshold for a
tight Hamilton cycle in a $k$-uniform hypergraph is $q_r$, which  settles, in this range, another conjecture of
Lang and Sanhueza-Matamala in 2022 stating that the lower-bound construction due to Han and Zhao is asymptotically optimal. Our proof combines an isoperimetric inequality for Johnson graphs with a Kruskal–Katona-type theorem, which is of independent interest.
\end{abstract}

\section{Introduction}\label{sec:intro}
The study of Hamilton cycles has been an important topic in graph and hypergraph theory. A celebrated result of
Dirac~\cite{Dirac52} states that every  graph $G$ on $n\ge3$ vertices with minimum degree $\delta(G)\ge n/2$ contains a Hamilton cycle. Determining analogous degree conditions for spanning structures
in hypergraphs has become a central problem in extremal combinatorics, and we refer to~\cite{RodlRucinski2010,KuhnOsthus2014,Zhao2016} for related surveys.

For an integer $k\ge2$, a \emph{$k$-uniform hypergraph}, or
\emph{$k$-graph}, $H$ consists of a vertex set $V(H)$ and an edge set
$E(H)\subseteq\binom{V(H)}k$. We write $e(H):=|E(H)|$
and $\mu(H):=e(H)/\binom {|V(H)|}k$ for its edge density. All subgraph densities are computed on the same vertex set unless stated otherwise.
A \emph{tight cycle} in $H$ is given by a cyclic ordering
$v_1,\ldots,v_t$ of distinct vertices, where $t\ge k+1$ and every $k$
consecutive vertices form an edge. It is \emph{Hamilton} if
$t=|V(H)|$. For $k=2$, this is the usual notion of a cycle in a graph. Katona and Kierstead~\cite{KatonaKierstead1999} initiated the
systematic study of degree conditions for tight Hamilton cycles.

For $1\le d<k$ and $S\in\binom{V(H)}d$,  the \emph{degree} of $S$ in $H$ is the number
of edges containing $S$, denoted by $\deg_H(S)$.
The \emph{minimum $d$-degree} of $H$ is denoted by
$\delta_d(H):=\min\{\deg_H(S):S\in\binom{V(H)}d\}$.
The case of $d=k-1$ is also known as the \emph{minimum codegree}. The
\emph{link} of $S$ is the $(k-d)$-graph $L_H(S)$ on $V(H)\setminus S$
whose edges are the sets $e\in\binom{V(H)\setminus S}{k-d}$ such that
$S\cup e\in E(H)$. Thus $e(L_H(S))=\deg_H(S)$. For a vertex $v$, we
write $L_H(v):=L_H(\{v\})$. For an $n$-vertex $k$-graph $H$, its \emph{shadow} $\bd H$ is the
family of $(k-1)$-sets contained in an edge of $H$, and we write
$\mu(\bd H):=|\bd H|/\binom n{k-1}$.

A tight Hamilton cycle is the case $\ell=k-1$ of a Hamilton
$\ell$-cycle, whose edges are intervals of $k$ vertices in a cyclic
ordering of all vertices, with successive edges overlapping in
$\ell$ vertices. For $1\le\ell<k$ with $(k-\ell)\nmid k$,
K\"uhn, Mycroft, and Osthus~\cite{KuhnMycroftOsthus2010}
determined the asymptotic minimum codegree threshold.
The exact threshold was subsequently obtained in complementary
ranges by Han and Zhao~\cite{HanZhao2015Codegree},
Gan, Han, and Xu~\cite{GanHanXu2026}, and Lu and
Yuan~\cite{LuYuan2026}. Results for lower minimum degrees include
those of Han, Sun, and Wang~\cite{HanSunWang2025} for $2$-cycles
in $5$-graphs and Lang and
Sanhueza-Matamala~\cite{LangSanhuezaLoose2025} for $(k-2)$-degree
conditions.

These questions also admit transversal variants; see~\cite{JoosKim2020, ChengHanWangWangYang2025,TangEtAl2026}. In such settings, a Hamilton cycle is formed by choosing one distinct edge from each member of
a hypergraph system on a common vertex set. Joos and
Kim~\cite{JoosKim2020} proved the exact transversal version of
Dirac's theorem. General methods for rainbow factors and cycles
were developed by Cheng, Han, Wang, and
Wang~\cite{ChengHanWangWang2023} and Gupta, Hamann, M{\"u}yesser, Parczyk and Sgueglia~\cite{GuptaEtAl2023}, respectively. 

\subsection{Dirac thresholds for tight Hamilton cycles}
Given integers $k>d\ge 1$, we define the \emph{Dirac threshold} $\delta_d^{\HC}(k)$  as the (asymptotic) minimum $d$-degree threshold for tight Hamilton
cycles. 
More precisely, $\delta_d^{\HC}(k)$ is the infimum of all $\delta\in[0,1]$ such that,
for every $\eps>0$, there exists $n_0\in 
\mathbb N$ for which every $n$-vertex
$k$-graph $H$ with $n\ge n_0$ and $\delta_d(H)\ge(\delta+\eps)\binom{n-d}{k-d}$
contains a tight Hamilton cycle.

Using the absorption method, R\"odl, Ruci\'nski, and
Szemer\'edi~\cite{RodlRucinskiSzemeredi2006,RodlRucinskiSzemeredi2008,RodlRucinskiSzemeredi2011}
established $\delta_{k-1}^{\HC}(k)=1/2$ for every $k\ge3$ and for $k = 3$ they obtained the exact value $\lfloor n/2\rfloor$.
Very recently, Lang, Letzter, Ranganathan, and Sanhueza-Matamala announced in a talk that, for every $k\ge3$ and sufficiently large $n$, the exact
minimum codegree threshold is
$\lceil(n-k+2)/2\rceil$, confirming a conjecture of Katona and Kierstead \cite{KatonaKierstead1999}.
For smaller degrees,
R\"odl, Ruci\'nski, Schacht, and
Szemer\'edi~\cite{RodlRucinskiSchachtSzemeredi2017} first showed that $\delta_{1}^{\HC}(3)\le 4/5$.
The equality $\delta_{k-2}^{\HC}(k)=5/9$ was then proved for $k=3$
by Reiher, R\"odl, Ruci\'nski, Schacht, and
Szemer\'edi~\cite{ReiherEtAl2019}, and then for $k=4$ by Polcyn,
Reiher, R\"odl, Ruci\'nski, Schacht, and
Sch\"ulke~\cite{PolcynEtAl2020}.
The general case was established by Polcyn, Reiher, R\"odl, and
Sch\"ulke~\cite{PolcynEtAl2021} and, independently, by Lang and
Sanhueza-Matamala~\cite{LangSanhueza2022}.
More recently, Lang, Schacht, and Volec~\cite{LangSchachtVolec} proved
that $\delta_{k-3}^{\HC}(k)=5/8$ for every $k\ge4$.

For general $d$, Han and Zhao~\cite{HanZhao2016} constructed
hypergraphs without tight Hamilton cycles that give systematic lower
bounds on $\delta_d^{\HC}(k)$. These lower bounds are conjectured to be asymptotically optimal
for every $1\le d<k$; see
\cite[Conjecture~11.2]{LangSanhueza2022} and
\cite[Conjecture~2.3]{LangSanhuezaBandwidth}.

\begin{conjecture}[{\cite{LangSanhueza2022,LangSanhuezaBandwidth}}]\label{conj:HC}
For all integers $r\ge1$ and $k\ge r+1$,
\[
\delta_{k-r}^{\HC}(k)
=
1-\binom r{\lfloor r/2\rfloor}
\frac{
\lceil r/2\rceil^{\lceil r/2\rceil}
(\lfloor r/2\rfloor+1)^{\lfloor r/2\rfloor}
}{
(r+1)^r
}.
\]
\end{conjecture}

For every odd integer $r\ge1$, set
\begin{equation}\label{eq:beta-def}
\beta_r:=2^{-r}\binom r{(r-1)/2},
\qquad
q_r:=1-\beta_r.
\end{equation}

For odd $r=k-d$, the right-hand side of the equation in~\cref{conj:HC} equals $q_r$ and
is witnessed by the following special case of the
construction of Han and Zhao~\cite{HanZhao2016}.

\begin{construction}[\cite{HanZhao2016}]\label{con:balanced}
Let $k\ge2$ and let $n\ge2k$ be even.
Partition $[n]=X\cup Y$ with $|X|=|Y|=n/2$.
Define the $k$-graph $H_{n,k}$ on $[n]$ by
$E(H_{n,k})=\left\{
e\in\binom{[n]}k:
|e\cap X|\ne\left\lceil k/2\right\rceil
\right\}.$
\end{construction}

As a special case of the construction of Han and
Zhao~\cite[Proposition~2.2]{HanZhao2016}, the $k$-graph $H_{n,k}$
contains no tight Hamilton cycle and for every odd
$r\in[k-1]$, $\delta_{k-r}(H_{n,k})
\ge
\left(q_r-O_{k,r}(n^{-1})\right)\binom{n-k+r}{r}$.
Consequently,
\begin{equation}\label{eq:lowerbound}
\delta_{k-r}^{\HC}(k)\ge q_r,
\qquad
\text{for every $k\ge r+1$ and odd $r\ge1$}.  
\end{equation}
Recently, Lang and Sanhueza-Matamala~\cite{LangSanhueza2022,LangSanhuezaBandwidth} developed a \emph{Hamilton framework} that reduces the corresponding upper bound problem to finding large
tight components and matchings in $(k-d)$-uniform links. 

\subsection{Large tight components}
Given a $k$-graph $H$, a \emph{(closed) tight walk} in $H$ is a (cyclically) ordered sequence of
vertices such that every interval of $k$ consecutive vertices forms
an edge of $H$. Vertices and edges may repeat. The \emph{length}
of a tight walk is its number of vertices.
A subgraph $H'$ of $H$ is \emph{tightly connected} if every two
edges of $H'$ lie on a common tight walk in $H'$.
A \emph{tight component} of $H$ is an edge-maximal tightly connected
subgraph. 

The connection between tight components and the Hamilton problem arises via links. A minimum
$d$-degree condition gives an edge density condition in every link of $d$-sets.
Let $H$ be an $n$-vertex $k$-graph, $r=k-d$ and $S\in \binom{V(H)}{d}$. If two
$r$-edges $e,f\in L_H(S)$ share $r-1$ vertices, then the
$k$-edges $S\cup e$ and $S\cup f$ share $k-1$ vertices. Thus, tight connectivity within a link provides the local tight transitions required to find tight cycles in $H$.
Moreover, for every fixed $\eta>0$ and sufficiently large $n$,
any two selected link components of density at least $1/2+\eta$ have an edge in common, and their edge-to-shadow ratios provide
matching estimates. These are the local properties used in the
Hamilton framework.

This leads to the following local extremal question: what edge density
is sufficient to force a tight component containing at least half of
all possible edges?
For graphs, this is easy to solve. However, for $k\ge 3$ much less is known. Allen, B\"ottcher, Cooley, and
Mycroft~\cite{AllenBottcherCooleyMycroft2017} obtained related results
in the sparse regime for $k=3$. Lang, Schacht, and
Volec~\cite{LangSchachtVolec} proved that edge density $5/8$
forces a tight component of asymptotic density $1/2$ in $3$-graphs. Motivated by
this, Lang and Sanhueza-Matamala~\cite[Conjecture~12.1]{LangSanhuezaBandwidth} 
made the following conjecture.

\begin{conjecture}[{\cite[Conjecture~12.1]{LangSanhuezaBandwidth} }]\label{conj:components}
For every odd integer $r\ge3$ and every $\eps>0$, there exists
$n_0=n_0(r,\eps)$ such that every $n$-vertex $r$-graph $H$ with
$n\ge n_0$ and $e(H)\ge q_r\binom nr$
contains a tight component $H_0$ satisfying $e(H_0)\ge\left(\frac12-\eps\right)\binom nr.$
\end{conjecture}

Note that the density is best possible, since 
the largest number of edges in a tight
component of $H_{n,k}$ is at most
$\frac12\binom nk$ for each $k\ge 2$.
Indeed, let $h=\lceil k/2\rceil$. Passing between two tightly adjacent
edges changes $|e\cap X|$ by at most one. Since edges with
$|e\cap X|=h$ are excluded, every tight component lies entirely in
one of the families defined by $|e\cap X|\le h-1$ and
$|e\cap X|\ge h+1$. By symmetry between $X$ and $Y$, each family
contains at most half of all $k$-sets.
For odd $r=2m+1\ge3$, this construction satisfies
\[
\mu(H_{n,r})
=1-\frac{\binom{n/2}{m+1}\binom{n/2}m}{\binom nr}
=q_r+O_r(n^{-1}).
\]

A related question asks for a tight component that spans all
vertices.
Georgakopoulos, Haslegrave, and
Montgomery~\cite{GeorgakopoulosHaslegraveMontgomery2019} investigated
how minimum codegree forces a tight component spanning many vertices
in a $3$-graph; in particular, minimum codegree at least
$\lfloor n/3\rfloor$ guarantees a spanning tight component.
More recently, Allsop, Lamaison, Lang, and
Rathke~\cite{AllsopLamaisonLangRathke} determined the asymptotic
minimum vertex-degree threshold for a spanning tight component in a
$3$-graph, while Di Braccio, Hearn, Lada, Neve, and
Zhang~\cite{DiBraccioEtAl} proved that every $n$-vertex $4$-graph
with minimum codegree at least $\lfloor n/4\rfloor$ has a spanning
tight component.
These results establish the degree conditions for a spanning
tight component.

\subsection{Our results}

Our first main result gives an explicit edge-density condition for a
large tight component in every uniformity. For every integer $k\ge2$,
define
\[
\widehat q_k:=
\begin{cases}
q_k, & \text{if $k$ is odd},\\
q_{k+1}, & \text{if $k$ is even}.
\end{cases}
\]

\begin{theorem}
\label{thm:components}
For every integer $k\ge2$ and every $\eps>0$, there exists
$n_0=n_0(k,\eps)$ such that every $k$-graph $H$ on $n\ge n_0$
vertices with $\mu(H)\ge\widehat q_k$ has a tight component $H_0$
satisfying $\mu(H_0)\ge1/2-\eps$.  
\end{theorem}

Thus, to our knowledge, \cref{thm:components} provides the first component-type result that holds for every uniformity $k\ge 2$.
In particular,  \cref{thm:components} confirms \cref{conj:components} for odd $k$.
Moreover, the construction in \cref{con:balanced} shows that both the density threshold $q_k$ and the component density $1/2$ are asymptotically best possible in this case. 
For even $k$, the theorem gives the explicit sufficient density $q_{k+1}$, which we do not claim is optimal.

In odd uniformity, 
\cref{thm:components} also implies a quantitative strengthening above
the threshold. A simple edge-deletion argument yields the following result.

\begin{corollary}\label{cor:components}
Let $r\ge3$ be odd and let $0<\eps<\beta_r$. There exists
$n_0=n_0(r,\eps)$ such that every $n$-vertex $r$-graph $H$ with
$n\ge n_0$ and $\mu(H)\ge q_r+\eps$
has a tight component $H_0$ satisfying
$\mu(H_0)\ge \frac12+\frac{\eps}{4}$.
\end{corollary}

As mentioned above, the positive margin above $1/2$ in \cref{cor:components} ensures that components selected
in different links intersect. As an application of~\cref{cor:components}, we prove~\cref{conj:HC} for all odd $r\ge 3$ and sufficiently large $k$. More precisely, for odd $r\ge 1$,  define
\begin{equation}\label{eq:cutoff}
 s_r:=\left\lceil r(3-2q_r)\right\rceil.
\end{equation}

\begin{theorem}\label{thm:main_all_odd}
For every odd integer $r\ge1$ and every integer $k\ge s_r$,
\begin{equation}\label{eq:main}
\delta_{k-r}^{\HC}(k)
=
q_r
=
1-2^{-r}\binom r{(r-1)/2}.
\end{equation}
\end{theorem}

The theorem provides a common formula for all larger odd link
uniformities in the stated range. For example, $\delta_{k-5}^{\HC}(k)=\frac{11}{16}$ for $k\ge9$ and $\delta_{k-7}^{\HC}(k)=\frac{93}{128}$ for $k\ge11$.
Since $\beta_r=O(r^{-1/2})$, we have
$s_r=r+O(\sqrt r)$ as $r$ tends to infinity through odd integers.
The restriction $k\ge s_r$ arises only when we want to find a sufficiently large matching in each selected link component; it is not required for applying
\cref{cor:components}. 

Note that for even $r$, some $r$-uniform
links in the construction of Han and Zhao~\cite{HanZhao2016} have
largest tight components of density less than $1/2-c_r$, for some constant $c_r>0$ depending on $r$.
Therefore, the corresponding Hamilton problem for even
link uniformity presents a different obstruction, although \cref{thm:components} applies to even 
uniformities as well. The compatibility argument used in this paper relies on link components of density greater than $1/2$ to ensure that components
selected in different links intersect. Extending
our result to even link uniformities will therefore require a
different method for selecting components with the necessary
intersection and matching properties.

\subsection{Proof overview and organization}
We first present the proof overview of \cref{thm:components},
then present the proof overview for \cref{thm:main_all_odd}, and finally describe the organization of the paper.

\paragraph{The large tight component theorem.} 
The proof of \cref{thm:components} combines a bound on neighborhoods in Johnson graphs with induction on the uniformity. The induction
first produces a component with a large shadow; the neighborhood
bound then forces this component to contain many edges.
We break the proof into the following steps.

\medskip\noindent
\emph{Step 1. Tight components and Johnson neighborhoods.}
Recall that the Johnson graph $J(n,k)$ has vertex set
$\binom{[n]}k$, with two $k$-sets adjacent when they intersect in
$k-1$ vertices. Thus, the tight components of an $n$-vertex
$k$-graph $H$ are precisely the connected components of the subgraph
of $J(n,k)$ induced by $E(H)$.

Let $H_0$ be a tight component of $H$. Every $k$-set outside
$E(H_0)$ that is adjacent in $J(n,k)$ to an edge of $H_0$ must be a
missing edge of $H$; otherwise it would belong to the same tight
component. Consequently,
\[
N_{\mathrm J}(E(H_0))
\subseteq
\binom{V(H)}k\setminus E(H).
\]
This allows us to compare a lower bound on the Johnson neighborhood
of $E(H_0)$ with the upper bound
$(1-\mu(H))\binom nk$ on the number of missing edges.

\medskip\noindent
\emph{Step 2. An isoperimetric estimate in the Johnson graph.}
To estimate Johnson neighborhoods, we replace each $k$-set by all of
its orderings and apply Harper's vertex-isoperimetric theorem for
Hamming graphs~\cite{Harper1999,AkhtarJiangMiller2008}. This yields a
continuous function $\Phi_k$ such that every family
$\cA\subseteq\binom{[n]}k$ of density
$x=|\cA|/\binom nk$ satisfies
\begin{equation}\label{eq:Johnson-neighborhood-inequality}
\frac{|N_{\mathrm J}(\cA)|}{\binom nk}
\ge \Phi_k(x)-o(1).
\end{equation}
Roughly, for sufficiently large $n$,  $\Phi_k(x)$ is the asymptotically smallest neighborhood density
among families of $k$-sets with density $x$;
its formula is given in \eqref{eq:Phi-def}.
The main calculation is to determine this bound at density $1/2$.
For odd $r\ge3$, we prove $\Phi_r(1/2)=1-\widehat q_r$.
Moreover, we prove a crucial property of $\Phi_k$, $\Phi_k(x)>1-\widehat q_k$ for  $\widehat q_k-\frac12<x<\frac12$.
These are the neighborhood estimates needed in the induction.

\medskip\noindent
\emph{Step 3. Induction through links and shadows.}
We prove \cref{thm:components} by induction on $k$, starting with an
elementary argument for graphs. For $k\ge3$, let $H$ be a
$k$-graph with $\mu(H)\ge\widehat q_k$. Averaging over vertex links
gives a vertex $v$ such that
$\mu(L_H(v))
\ge \mu(H)
\ge \widehat q_k
\ge \widehat q_{k-1}$.
By induction, $L_H(v)$ contains a tight component with density
$1/2-o(1)$. The corresponding edges of $H$ lie in a single tight
component $H_0$, and hence $\mu(\bd H_0)\ge 1/2-o(1)$.

Let $H'$ be the
spanning subgraph with edge set $E(H)\setminus E(H_0)$. Since
distinct tight components have disjoint shadows,
$\mu(\bd H')\le1/2+o(1)$. Then the Kruskal--Katona
theorem~\cite{Kruskal1963} gives
$\mu(H')\le2^{-k/(k-1)}+o(1)<1/2$.
Unless $H_0$ already has density $1/2-o(1)$, these estimates place
$\mu(H_0)$ strictly between $\widehat q_k-1/2$ and $1/2$.
By Step~2, the Johnson neighborhood of $E(H_0)$ then has density
strictly greater than $1-\widehat q_k$. Step~1 shows that all these
sets are missing edges of $H$, contradicting
$\mu(H)\ge\widehat q_k$.

\paragraph{The tight Hamilton cycle theorem.}
For the proof of~\cref{thm:main_all_odd}, we use the
\emph{Hamilton framework} of Lang and
Sanhueza-Matamala~\cite{LangSanhueza2022,LangSanhuezaBandwidth},
which reduces the upper-bound problem to showing that the family of
hypergraphs satisfying the degree condition admits a Hamilton
framework. Let $r\ge3$ be odd, set $d=k-r$, and suppose that $H$ is a $k$-graph with $\delta_d(H)\ge(q_r+\eps)\binom{n-d}{r}$.
Then every $d$-link has density at least $q_r+\eps$. For each
$S\in\binom{V(H)}d$, we select the tight component $H_S$ of
$L_H(S)$ supplied by \cref{cor:components}. Since $\mu(H_S)\ge 1/2+\eps/4$,
any two selected components share an edge, providing the compatibility needed to combine the local
choices.
It remains to obtain the matching structure required by the Hamilton framework. 
Disjointness of component shadows yields a suitable lower bound on
$\mu(H_S)/\mu(\bd H_S)$. 
We apply Frankl's shadow--matching
inequality~\cite{Frankl2013} to obtain a fractional matching of size $n/k$ in each $H_S$, provided that $k\ge s_r$. A matching
transfer lemma (\cref{lem:matching-transfer}) then yields a perfect fractional matching in the
$k$-graph generated by the selected components.
The component
densities and the edge-to-shadow estimates, together with the
vicinity lemmas of Lang and
Sanhueza-Matamala~\cite{LangSanhuezaBandwidth}, also provide the
required switchers and an arc; see~\cref{lem:vicinity-toolkit}. These choices give a Hamilton framework for the family of $k$-graphs satisfying the degree condition and the bandwidth
theorem~\cite[Theorem~1.1]{LangSanhuezaBandwidth} gives $\delta_{k-r}^{\HC}(k)\le q_r$.
Together with the Han--Zhao construction, this proves
\cref{thm:main_all_odd}.

\paragraph{Organization.}
In \cref{sec:profile}, we define $\Phi_k$ and establish the Johnson
neighborhood inequality. In \cref{sec:analytic}, we prove that
$\Phi_r(1/2)=\beta_r$ for every odd $r\ge3$.
We prove \cref{thm:components,cor:components} in
\cref{sec:components} and \cref{thm:main_all_odd} in
\cref{sec:hamilton}.
We conclude with further questions in \cref{sec:conclusion}.

\section{Neighborhoods in Hamming and Johnson graphs}
\label{sec:profile}
Throughout this section, let $k\ge2$ be an integer and write
$[n]=\{1,\ldots,n\}$.
Recall that the \emph{Johnson graph} $J(n,k)$ has vertex set
$\binom{[n]}k$, with two vertices adjacent when their intersection has
size $k-1$. The \emph{Hamming graph} $\cH_{n,k}=K_n^{\square k}$ has
vertex set $[n]^k$, with two vertices adjacent when they differ in
exactly one coordinate. For a family $\cA$ of vertices in either graph,
we write $N_{\mathrm J}(\cA)$ or $N_{\mathrm H}(\cA)$ for its external
neighborhood, excluding $\cA$ itself. In particular,
\[
N_{\mathrm J}(\cA)=
\left\{e\in\binom{[n]}k\setminus\cA:
 |e\cap f|=k-1\text{ for some }f\in\cA\right\}.
\]

Our goal is to show that every family
$\cA\subseteq\binom{[n]}k$ of density $x\in(0,1)$ has a neighborhood
of density at least $\Phi_k(x)-o(1)$. We first define $\Phi_k$ using
weighted Hamming balls in the hypercube. We then use Harper's
vertex-isoperimetric theorem for Hamming graphs, in the finite-graph
form recorded in \cite[Proposition~6.1 and
Theorem~6.3]{AkhtarJiangMiller2008}, and finally pass from $k$-tuples to $k$-sets.

\subsection{The function \texorpdfstring{$\Phi_k$}{Phi-k}}

Let $Q_k=\{0,1\}^k$ be the hypercube. Recall that the \emph{Hamming weight} of
$\mathbf z=(z_1,\ldots,z_k)\in Q_k$ is $|\mathbf z|:=|\{i\in[k]:z_i=1\}|$.
 For $0\le t\le1$, and
$\mathbf z\in Q_k$, define the vertex weight $
w_t(\mathbf z)
:=t^{k-|\mathbf z|}(1-t)^{|\mathbf z|}.$
For $A\subseteq Q_k$, write $w_t(A):=\sum_{\mathbf z\in A}w_t(\mathbf z).$
Thus $w_t(Q_k)=1$.
For an integer $0\le s\le k-1$, let
\[
H_{k,s}:=\{\mathbf z\in Q_k:|\mathbf z|\le s\}.
\]
Thus $H_{k,s}$, equipped with the vertex weighting $w_t$, is a
\emph{weighted Hamming ball} of radius $s$ centered at $\mathbf0$. Write $N_{Q_k}$ for the neighborhood in the
hypercube. Then $N_{Q_k}(H_{k,s})$ consists
precisely of the vectors of Hamming weight $s+1$. Hence
\begin{align}
V_{k,s}(t)&:=w_t(H_{k,s})
 =\sum_{i=0}^s\binom ki t^{k-i}(1-t)^i,
\label{eq:V-def}\\
B_{k,s}(t)&:=w_t(N_{Q_k}(H_{k,s}))
 =\binom k{s+1}t^{k-s-1}(1-t)^{s+1}.
\label{eq:B-def}
\end{align}
For $0\le s\le k-1$, the function $V_{k,s}$ increases continuously
from $0$ to $1$, since
\begin{equation}\label{eq:V-derivative}
V'_{k,s}(t)=k\binom{k-1}s t^{k-s-1}(1-t)^s>0
\qquad(0<t<1).
\end{equation}
Hence, for every $x\in(0,1)$, there is a unique
$t_{k,s}(x)\in(0,1)$ such that
$V_{k,s}(t_{k,s}(x))=x$. The weighted Hamming ball $H_{k,s}$ has total weight $x$ and its neighborhood has total
weight $B_{k,s}(t_{k,s}(x))$.
The case $s=k-1$ has
$V_{k,k-1}(t)=1-(1-t)^k$ and
$B_{k,k-1}(t)=(1-t)^k$, so its neighborhood has total weight $1-x$. It can be omitted when
minimizing, since
\[
B_{k,0}(t)=kt^{k-1}(1-t)
\le (1-t)\sum_{i=0}^{k-1}t^i
=1-t^k=1-V_{k,0}(t).
\]
We therefore define
\begin{equation}\label{eq:Phi-def}
\Phi_k(x):=\min_{0\le s\le k-2}
 B_{k,s}\bigl(t_{k,s}(x)\bigr)
\qquad(0<x<1).
\end{equation}
Thus $\Phi_k(x)$ is the smallest neighborhood weight among weighted
Hamming balls of total weight $x$.
Each function $x\mapsto B_{k,s}(t_{k,s}(x))$ is continuous on
$(0,1)$, so $\Phi_k$ is continuous there. Taking $s=0$ gives $ 0\le\Phi_k(x)\le kx^{(k-1)/k}(1-x^{1/k}).$
Setting $\Phi_k(0)=\Phi_k(1)=0$ therefore makes $\Phi_k$ continuous,
and hence uniformly continuous, on $[0,1]$.

We now consider weighted Hamming balls in the finite Hamming graph
$\cH_{n,k}$. For a set $X\subseteq[n]$ with $|X|=m$ and an integer
$0\le s\le k-1$, define the \emph{weighted Hamming ball}
\begin{equation}\label{eq:finite-weighted-ball}
\cB_{n,k,s}(X):=
\left\{\mathbf v=(v_1,\ldots,v_k)\in[n]^k:
 |\{i\in[k]:v_i\notin X\}|\le s\right\}.
\end{equation}
Writing $t=m/n$, for $1\le m\le n-1$, its neighborhood
consists precisely of the tuples with exactly $s+1$ coordinates
outside $X$. Hence
\begin{equation}\label{eq:finite-ball-counts}
\frac{|\cB_{n,k,s}(X)|}{n^k}=V_{k,s}(m/n),
\qquad
\frac{|N_{\mathrm H}(\cB_{n,k,s}(X))|}{n^k}=B_{k,s}(m/n).
\end{equation}

Harper's vertex-isoperimetric theorem for Hamming graphs gives the
following lower bound; see \cite[Theorem~1]{Harper1999} and
\cite[Proposition~6.1 and Theorem~6.3]{AkhtarJiangMiller2008}.

\begin{theorem}[{\cite{Harper1999,AkhtarJiangMiller2008}}]
\label{thm:Hamming-Harper}
For all integers $n,k\ge2$ and every $\F\subseteq[n]^k$,
\begin{equation}\label{eq:Hamming-boundary}
\frac{|N_{\mathrm H}(\F)|}{n^k}
\ge
\Phi_k\left(\frac{|\F|}{n^k}\right).
\end{equation}
\end{theorem}

We now prove the Johnson neighborhood inequality~\eqref{eq:Johnson-neighborhood-inequality}
from \cref{thm:Hamming-Harper} by passing from $k$-sets to their
orderings in the Hamming graph.

\begin{lemma}\label{lem:Johnson}
For every integer $k\ge2$ and every $\eps>0$, there exists
$n_0\in \mathbb N$ such that the following holds for all $n\ge n_0$.
For every $\varnothing\ne\cA\subsetneq\binom{[n]}k$, writing
$x:=|\cA|/\binom nk$, we have
\begin{equation}\label{eq:Johnson-profile}
\frac{|N_{\mathrm J}(\cA)|}{\binom nk}\ge\Phi_k(x)-\eps.
\end{equation}
\end{lemma}

\begin{proof}
Let $\cA\subseteq\binom{[n]}k$ be as in the statement, and let
\[
\widehat{\cA}:=\{(v_1,\ldots,v_k)\in[n]^k:
 v_1,\ldots,v_k\text{ are distinct and }
 \{v_1,\ldots,v_k\}\in\cA\}.
\]
Then $|\widehat{\cA}|=k!|\cA|$.
The tuples with distinct entries in $N_{\mathrm H}(\widehat{\cA})$
are precisely the orderings of members of $N_{\mathrm J}(\cA)$.
There are at most $\binom k2 n^{k-1}$ tuples with repeated entries.
Applying \cref{thm:Hamming-Harper} to $\widehat{\cA}$, we obtain
\[
n^k\Phi_k\left(\frac{k!|\cA|}{n^k}\right)
\le |N_{\mathrm H}(\widehat{\cA})|
\le k!|N_{\mathrm J}(\cA)|+\binom k2 n^{k-1}.
\]
Since $k!\binom nk\le n^k$, it follows that
\[
\frac{|N_{\mathrm J}(\cA)|}{\binom nk}
\ge\Phi_k\left(\frac{k!|\cA|}{n^k}\right)-O_k(n^{-1}).
\]
Moreover, $k!|\cA|/n^k=x+O_k(n^{-1})$, uniformly over $x\in(0,1)$.
The uniform continuity of $\Phi_k$ on $[0,1]$ implies that the
last expression is at least $\Phi_k(x)-\eps$ when $n$ is sufficiently large.
\end{proof}

The bound in \cref{lem:Johnson} is asymptotically sharp.
Fix $x\in(0,1)$, choose $s$ attaining the minimum in
\eqref{eq:Phi-def}, and let $t=t_{k,s}(x)$.
Take $X\subseteq[n]$ with $|X|=\lfloor tn\rfloor$ and let
$\cA=\{e\in\binom{[n]}k:|e\setminus X|\le s\}$.
For sufficiently large $n$, its neighborhood consists precisely
of the $k$-sets with $s+1$ vertices outside $X$.
Thus $|\cA|/\binom nk\to V_{k,s}(t)=x$ and
$|N_{\mathrm J}(\cA)|/\binom nk\to B_{k,s}(t)=\Phi_k(x)$
as $n\to\infty$.

\subsection{Two inequalities for \texorpdfstring{$\Phi_k$}{Phi-k}}
We need two further inequalities in \cref{sec:components}.
The first is used in the even steps of the induction.

\begin{lemma}\label{lem:cylinder}
For every integer $k\ge2$ and every $x\in(0,1)$,
$\Phi_k(x)\ge\Phi_{k+1}(x)$.
\end{lemma}

\begin{proof}
Choose $s$ attaining the minimum in \eqref{eq:Phi-def}, and let
$t=t_{k,s}(x)$. For each $n$, take $X_n\subseteq[n]$ with
$|X_n|=\lfloor tn\rfloor$ and let $\cB_n:=\cB_{n,k,s}(X_n)\subseteq[n]^k.$
By \eqref{eq:finite-ball-counts},
\[
\frac{|\cB_n|}{n^k}\to V_{k,s}(t)=x
~~\text{and}~~
\frac{|N_{\mathrm H}(\cB_n)|}{n^k}
\to B_{k,s}(t)=\Phi_k(x).
\]
Now let $\cC_n:=\cB_n\times[n]\subseteq[n]^{k+1}.$
Then $|\cC_n|/n^{k+1}=|\cB_n|/n^k$
and $N_{\mathrm H}(\cC_n)=N_{\mathrm H}(\cB_n)\times[n],$ since changing only the last coordinate cannot enter or leave
$\cB_n\times[n]$. Applying \cref{thm:Hamming-Harper} in dimension
$k+1$ gives
\[
\Phi_{k+1}\left(\frac{|\cB_n|}{n^k}\right)
\le
\frac{|N_{\mathrm H}(\cB_n)|}{n^k}.
\]
Letting $n\to\infty$ and using the continuity of $\Phi_{k+1}$ yields
$\Phi_k(x)\ge\Phi_{k+1}(x)$.
\end{proof}

The next lemma gives a lower bound on $\Phi_k(x)$ for
$1/2-\beta\le x\le1/2$, assuming that $\Phi_k(1/2)\ge\beta$.

\begin{lemma}\label{lem:reflection}
Suppose $k\ge2$, $0<\beta<1/2$, and $\Phi_k(1/2)\ge\beta$.
For $1/2-\beta\le x\le1/2$, we have
\begin{equation}\label{eq:reflection-quantitative}
\Phi_k(x)\ge\beta+
\frac{(x-1/2+\beta)(1/2-x)}{2k(k-1)}.
\end{equation}
In particular, $\Phi_k(x)>\beta$ whenever
$1/2-\beta<x<1/2$.
\end{lemma}

\begin{proof}
Fix $0\le s\le k-2$, let $s'=k-s-2$, and write
$f(x)=B_{k,s}(t_{k,s}(x))$.
By \eqref{eq:V-def} and \eqref{eq:B-def}, we have 
$B_{k,s}(t)=B_{k,s'}(1-t)$ and $V_{k,s}(t)+B_{k,s}(t)+V_{k,s'}(1-t)=1$.

Choose $u\in(0,1)$ such that $V_{k,s'}(1-u)=1/2$, and let
$x_0=V_{k,s}(u)$.
Then $t_{k,s}(x_0)=u$ and $t_{k,s'}(1/2)=1-u$.
The first identity and the definition of $\Phi_k$ give
\[
f(x_0)=B_{k,s}(u)=B_{k,s'}(1-u)\ge\Phi_k(1/2)\ge\beta.
\]
The second identity gives $x_0+f(x_0)+1/2=1$, so
$x_0=1/2-f(x_0)\le1/2-\beta$.
Moreover, $f(1/2)\ge\Phi_k(1/2)\ge\beta$.

Writing $t=t_{k,s}(x)$ and differentiating gives
\[
f''(x)=-\frac{k-s-1}{(s+1)t^2V'_{k,s}(t)}
\le-\frac1{k(k-1)},
\]
where we use $V'_{k,s}(t)\le k$ from \eqref{eq:V-derivative},
$t\le1$, and $(k-s-1)/(s+1)\ge1/(k-1)$.
Thus, $f(x)+x^2/(2k(k-1))$ is concave. Since
$f(x_0),f(1/2)\ge\beta$ and $x_0\le1/2-\beta$, concavity gives
\[
\begin{aligned}
f(x)&\ge\beta+\frac{(x-x_0)(1/2-x)}{2k(k-1)}\\
&\ge\beta+\frac{(x-1/2+\beta)(1/2-x)}{2k(k-1)}
\end{aligned}
\]
for $1/2-\beta\le x\le1/2$.
Taking the minimum over $s$ proves \eqref{eq:reflection-quantitative}.
\end{proof}

\section{The neighborhood volume at density \texorpdfstring{$1/2$}{1/2}}
\label{sec:analytic}

In this section, we determine $\Phi_r(1/2)$ for odd $r$.

\begin{theorem}\label{thm:half}
For every odd $r=2m+1\ge3$,
\[
\Phi_r(1/2)=2^{-r}\binom rm=\beta_r.
\]
\end{theorem}

Fix $r=2m+1\ge3$. By \eqref{eq:V-def} and \eqref{eq:B-def}, $V_{r,m}(1/2)=\frac12
~\text{and}~
B_{r,m}(1/2)=2^{-r}\binom r{m+1}=\beta_r.$
Thus $t_{r,m}(1/2)=1/2$ and $\Phi_r(1/2)\le\beta_r$.
To prove the reverse inequality, we compare the neighborhood
weights of all weighted Hamming balls of total weight $1/2$.

For $0\le s\le r-2$, let $p_{r,s}:=1-t_{r,s}(1/2)$, and write
$\Bin(r,p)$ for a binomial random variable with parameters $r$ and $p$.
By the definition of $t_{r,s}(1/2)$,
$\Pr(\Bin(r,p_{r,s})\le s) =V_{r,s}(1-p_{r,s})=\frac12.$
Define
\[
b_{r,s}:=B_{r,s}(t_{r,s}(1/2))
=\Pr(\Bin(r,p_{r,s})=s+1)=\binom r{s+1}p_{r,s}^{\,s+1}(1-p_{r,s})^{r-s-1}.
\]
Consequently,
\begin{equation}\label{eq:half-binomial}
\Phi_r(1/2)=\min_{0\le s\le r-2}b_{r,s}.
\end{equation}
We have $p_{r,m}=1/2$ and $b_{r,m}=\beta_r$.
We will show that $b_{r,s}>\beta_r$ whenever $s\ne m$.

To estimate $p_{r,s}$, let $U$ have the beta distribution
$\Beta(s+1,r-s)$, whose density on $(0,1)$ is proportional to
$u^s(1-u)^{r-s-1}$.
For every $p\in(0,1)$, we have
\[
\Pr(\Bin(r,p)\ge s+1)
=\frac{\int_0^p u^s(1-u)^{r-s-1}\,du}
       {\int_0^1 u^s(1-u)^{r-s-1}\,du}
=\Pr(U\le p).
\]
By the definition of $p_{r,s}$, we have
$\Pr(U\le p_{r,s})=1/2$, so $p_{r,s}$ is the median of $U$.
For $0\le s\le m$, the median--mean inequality for beta
distributions gives $p_{r,s}\le(s+1)/(r+1)$; see
\cite[Corollary~7]{ArabOliveiraWiklund2021}.
We combine this with the lower bound proved below.

\begin{lemma}\label{lem:median}
For $r=2m+1\ge3$ and $0\le s\le m$,
\[
\frac{s+1/2}{r}\le p_{r,s}\le\frac{s+1}{r+1}.
\]
\end{lemma}

\begin{proof}
It remains to prove the lower bound.
Let $U\sim\Beta(s+1,r-s)$ and let $\Theta=\arcsin\sqrt U$.
The density of $\Theta$ on $(0,\pi/2)$ is proportional to $h(\theta)=\sin^{2s+1}\theta\cos^{2r-2s-1}\theta$.
Choose $\theta_0\in(0,\pi/4]$ such that
$u:=\tan\theta_0=\sqrt{(2s+1)/(2r-2s-1)}$.
Since $\sin^2\theta_0=(s+1/2)/r$, it suffices to show that
$\Pr(\Theta\ge\theta_0)\ge1/2$.

For $0<t<\theta_0$, let $v=\tan t$, so $0<v<u\le1$.
The addition formulas give
\[
\log\frac{h(\theta_0+t)}{h(\theta_0-t)}
=(2r-2s-1)\left(u^2\log\frac{u+v}{u-v}
                  -\log\frac{1+uv}{1-uv}\right)\ge0.
\]
Indeed, the function $u^2\log\frac{u+v}{u-v}
                  -\log\frac{1+uv}{1-uv}$ vanishes at $v=0$, and its
derivative with respect to $v$ is
\[
\frac{2uv^2(1-u^4)}{(u^2-v^2)(1-u^2v^2)}\ge0.
\]
Thus, $h(\theta_0+t)\ge h(\theta_0-t)$ for $0<t<\theta_0$.
Since $2\theta_0\le\pi/2$, we obtain
\[
\int_0^{\theta_0}h(\theta)\,d\theta
\le\int_{\theta_0}^{2\theta_0}h(\theta)\,d\theta
\le\int_{\theta_0}^{\pi/2}h(\theta)\,d\theta.
\]
Hence $\Pr(\Theta\ge\theta_0)\ge1/2$, and therefore
$p_{r,s}\ge\sin^2\theta_0=(s+1/2)/r$.
\end{proof}

\begin{proof}[Proof of \cref{thm:half}]
By \eqref{eq:half-binomial}, it remains to prove $b_{r,s}>\beta_r$
for $s\ne m$. We consider the three cases below.

\medskip\noindent\emph{Case 1: $0\le s\le m-2$.}
Here $r\ge5$. For $0\le s\le m-1$, let
\[
f_r(s):=\binom r{s+1}
 \left(\frac{s+1/2}{r}\right)^{s+1}
 \left(\frac{r-s-1/2}{r}\right)^{r-s-1}.
\]
By \cref{lem:median},
$(s+1/2)/r\le p_{r,s}\le(s+1)/(r+1)<(s+1)/r$.
Since $\binom r{s+1}p^{s+1}(1-p)^{r-s-1}$ is increasing for
$0<p<(s+1)/r$, we have $b_{r,s}\ge f_r(s)$.
Next, we show that $f_r(i-1)>f_r(i)$ for each $i\in [m-2]$ 
and $f_r(m-2)>\beta_r$.

For $0\le s\le m-2$, let $\ell=r-s-1\ge s+4$. Taking logarithms gives
\[
\begin{aligned}
\log\frac{f_r(s+1)}{f_r(s)}
= &(s+1)\log\frac{s+3/2}{s+1/2}
 +\log\left(1-\frac1{2(s+2)}\right)-(\ell-1)\log\frac{\ell+1/2}{\ell-1/2}
 -\log\left(1+\frac1{2\ell}\right)\\
\le{}&\frac1{3(4(s+1)^2-1)}-\frac1{2(s+2)}
       +\frac1{2\ell}+\frac1{8\ell^2}.
\end{aligned}
\]
The inequality follows from the power series for
$\log((1+y)/(1-y))$, together with $\log(1-y)\le-y$ and
$\log(1+y)\ge y-y^2/2$ for $0<y<1$.

For $s\ge1$, we have $2(s+2)(s+3)\le3(4(s+1)^2-1)$.
Moreover, $\frac1{2\ell}+\frac1{8\ell^2}
\le\frac1{2\ell}+\frac1{2\ell(2\ell-1)}
=\frac1{2\ell-1}.$
Combining these estimates and using $\ell\ge s+4$, we obtain
\[
\log\frac{f_r(s+1)}{f_r(s)}
\le
\frac1{2(s+2)(s+3)}-\frac1{2(s+2)}
+\frac1{2\ell-1}\le-\frac1{2(s+3)}+\frac1{2s+7}<0.
\]
For $s=0$, we have $\ell\ge4$, and 
$\log\frac{f_r(1)}{f_r(0)}\le 1/9-1/4+1/8+1/128=-7/1152<0$.

Recalling $r=2m+1$ and $\beta_r=2^{-r}\binom r{(r-1)/2}$, we have
\[
\frac{f_r(m-2)}{\beta_r}
=\frac{r-1}{r+3}
 (1-4/r)^{(r-3)/2}(1+4/r)^{(r+3)/2}.
\]
For $r=5$, this ratio is $6561/6250>1$.
For $r\ge7$, let $x:=1/r$ and
\[
F(x):=\log\frac{1-x}{1+3x}
 +\frac{\log(1-16x^2)}{2x}
 +\frac32\log\frac{1+4x}{1-4x}.
\]
Then $F(1/r)=\log(f_r(m-2)/\beta_r)$ and $F(x)\to0$ as $x\to0$. Moreover, 
\[
\begin{aligned}
F'(x)
&=-\frac4{(1-x)(1+3x)}-\frac4{1-16x^2}
  -\frac{\log(1-16x^2)}{2x^2}\\
&\ge
\frac{4x(2-3x-32x^2-208x^3-512x^4+768x^5)}
 {(1-x)(1+3x)(1-16x^2)}>0,
\end{aligned}
\]
where the first inequality follows from
$-\log(1-y)\ge y+y^2/2$ and the last inequality holds for $0<x\le 1/7$.
Hence $f_r(m-2)>\beta_r$.

\medskip\noindent\emph{Case 2: $s=m-1$.} Recall that $b_{r,m-1}=\binom r{m}p_{r,m-1}^{\,m}(1-p_{r,m-1})^{r-m}$.
We first prove that
\begin{equation}\label{eq:m-1}
p_{r,m-1}\ge\frac12-\frac1{r+1/4}.    
\end{equation}

Let $U\sim\Beta(m,m+2)$ and let $Z=1-2U$.
The density of $Z$ is proportional to
$h(z)=(1-z^2)^{m-1}(1+z)^2$ on $(-1,1)$.
Since
\[
\frac12\int_0^1\bigl(h(z)-h(-z)\bigr)\,dz
=2\int_0^1z(1-z^2)^{m-1}\,dz=\frac1m,
\]
the median of $Z$ is the unique positive $z$ satisfying
$\int_0^z h(t)\,dt=1/m$.
Set $z_0=8/(4r+1)$. To bound this median from above by $z_0$, it
suffices to prove $\int_0^{z_0}h(z)\,dz\ge1/m$.
Bernoulli's inequality gives
$(1-z^2)^{m-1}\ge1-(m-1)z^2$ for $0\le z\le z_0$.
Multiplying by $(1+z)^2$ and integrating, we obtain
\[
\begin{aligned}
\int_0^{z_0}h(z)\,dz
&\ge z_0+z_0^2+\frac{2-m}{3}z_0^3
       -\frac{m-1}{2}z_0^4-\frac{m-1}{5}z_0^5\\
&=\frac1m+
\frac{2(1280r^4-1280r^3+45024r^2+8624r-100523)}
     {15(r-1)(4r+1)^5}>\frac1m.
\end{aligned}
\]
The equality follows by substituting $m=(r-1)/2$ and
$z_0=8/(4r+1)$. The numerator of the additional term is positive,
since $1280r^4-1280r^3>0$ and $45024r^2>100523$ for $r\ge3$.
Thus the median of $Z$ is less than $z_0$.
Since this median equals $1-2p_{r,m-1}$, we obtain~\eqref{eq:m-1}.

For $r\ge5$, by~\eqref{eq:m-1} and $r/4\ge1$, we have $p_{r,m-1}\ge1/2-r/(r^2+1)$. Moreover, $p_{r,m-1}\le m/(r+1)<m/r$ by \cref{lem:median}.
The probability $\binom rm p^m(1-p)^{r-m}$ increases for $0<p<m/r$,
so
\[
\frac{b_{r,m-1}}{\beta_r}
\ge\frac{(r-1)^{r-1}(r+1)^{r+1}}{(r^2+1)^r}>1.
\]
For $r=3$, \eqref{eq:m-1} gives $p_{3,0}\ge5/26$, and the
same monotonicity yields
\[
b_{3,0}=\binom{3}{1}p_{3,0}(1-p_{3,0})^2\ge3\cdot\frac5{26}\left(\frac{21}{26}\right)^2
>\frac38=\beta_3.
\]
Hence, $b_{r,m-1}>\beta_r$ for every odd $r\ge3$.

\medskip\noindent\emph{Case 3: $m+1\le s\le r-2$.}
Let $s'=r-s-1$. Then $1\le s'\le m-1$ and $b_{r,s'}>\beta_r$.
If $U\sim\Beta(s+1,r-s)$, then $1-U\sim\Beta(r-s,s+1)$.
Since $s'+1=r-s$, their medians satisfy
$p_{r,s}=1-p_{r,s'}$. Consequently,
\[
\frac{b_{r,s}}{b_{r,s'}}
=\frac{(s'+1)(1-p_{r,s'})}{(r-s')p_{r,s'}}\ge1,
\]
where the inequality follows from
$p_{r,s'}\le(s'+1)/(r+1)$ in \cref{lem:median}. Thus, $b_{r,s}\ge b_{r,s'}>\beta_r$.

Together with $b_{r,m}=\beta_r$, these inequalities prove the theorem
by \eqref{eq:half-binomial}.
\end{proof}

\section{Proof of large tight components}
\label{sec:components}

We use the neighborhood estimates from the preceding sections
to prove \cref{thm:components,cor:components}. The induction also
needs the following normalized form of the Kruskal--Katona
shadow bound~\cite{Kruskal1963}.

\begin{lemma}\label{lem:KK}
Let $n\ge k\ge2$. Every $n$-vertex $k$-graph $H$ satisfies
\[
\mu(H)\le\mu(\bd H)^{k/(k-1)}.
\]
\end{lemma}

\begin{proof}
There is nothing to prove if $e(H)=0$.
Otherwise, choose a real $k\le x \le n$ such that $e(H)=\binom xk$.
By the Kruskal--Katona theorem,
$|\bd H|\ge\binom{x}{k-1}$.
Since $x\le n$, the sequence $(x-i)/(n-i)$ is nonincreasing in $i$.
Consequently,
\[
\left(\frac{\binom xk}{\binom nk}\right)^{1/k}
\le
\left(\frac{\binom x{k-1}}{\binom n{k-1}}\right)^{1/(k-1)}.
\]
It follows that
\[
\mu(H)
\le
\left(\frac{\binom x{k-1}}{\binom n{k-1}}\right)^{k/(k-1)}
\le
\mu(\bd H)^{k/(k-1)}.
\]
\end{proof}

\begin{proof}[Proof of~\cref{thm:components}]
We proceed by induction on $k$.
For $k=2$, we have $\widehat q_2=5/8$.
Let $H$ be a graph on $n\ge4$ vertices with
$e(H)\ge(5/8)\binom n2$.
Let $X$ be the vertex set of a connected component with most
vertices, and set $m=|X|$.
Every vertex has degree at most $m-1$, so
$2e(H)\le n(m-1)$ and hence $m>n/2$.
If $m\le3n/4$, then 
\[
e(H)\le\binom m2+\binom{n-m}{2}
\le\frac{5n^2}{16}-\frac n2
<\frac58\binom n2,
\]
a contradiction. Thus $m>3n/4$.
Since all edges outside $H[X]$ lie in $V(H)\setminus X$, we obtain
\[
e(H[X])\ge e(H)-\binom{n-m}{2}
>\frac58\binom n2-\frac{n^2}{32}
\ge\frac12\binom n2.
\]
This proves the base case.

Now let $k\ge3$ and assume that the statement holds for $k-1$.
By decreasing $\eps$ if necessary, we may assume that
$0<\eps<(1-\widehat q_k)/2$ and
\[
\left(\frac12+2\eps\right)^{k/(k-1)}
<\frac12-\eps.
\]
Such a choice is possible since $2^{-k/(k-1)}<1/2$.
Choose $n_0\ge k/\eps$ sufficiently large that the induction
hypothesis applies to $(k-1)$-graphs on $n-1$ vertices with error
$\eps$, and that \cref{lem:Johnson} applies with error
$\eps^2/(4k(k-1))$, whenever $n\ge n_0$.

Let $H$ be a $k$-graph on $n\ge n_0$ vertices with
$\mu(H)\ge\widehat q_k$.
Suppose, for a contradiction, that every tight component of $H$
has density less than $1/2-\eps$.
Averaging over vertex links gives a vertex $v$ such that
\[
\mu(L_H(v))\ge\mu(H)\ge\widehat q_k\ge\widehat q_{k-1},
\]
where the last inequality follows from \eqref{eq:beta-def}.
By induction, $L_H(v)$ has a tight component $H_v$ with at least
$(1/2-\eps)\binom{n-1}{k-1}$ edges.
Since $H_v$ is tightly connected, all edges $e\cup\{v\}$ with
$e\in E(H_v)$ lie in a single tight component $H_0$ of $H$.
Thus $E(H_v)\subseteq\bd H_0$, and
\[
\mu(\bd H_0)
\ge\left(\frac12-\eps\right)
   \frac{\binom{n-1}{k-1}}{\binom n{k-1}}
=\left(\frac12-\eps\right)\frac{n-k+1}{n}
\ge\frac12-2\eps.
\]

Let $H'$ be the spanning subgraph of $H$ with edge set
$E(H)\setminus E(H_0)$.
Distinct tight components have disjoint shadows, so
$\mu(\bd H')\le1/2+2\eps$.
By \cref{lem:KK} and our choice of $\eps$,
\[
\mu(H')\le\mu(\bd H')^{k/(k-1)}
\le\left(\frac12+2\eps\right)^{k/(k-1)}
<\frac12-\eps.
\]
Consequently,
\[
\widehat q_k-\frac12+\eps
<\mu(H_0)=\mu(H)-\mu(H')
<\frac12-\eps,
\]
where the upper bound follows from our assumption on the components.

By \cref{thm:half,lem:cylinder}, we have
$\Phi_k(1/2)\ge1-\widehat q_k$.
Applying \cref{lem:reflection} with $\beta=1-\widehat q_k$
therefore gives
\[
\Phi_k(\mu(H_0))
>\beta +\frac{\eps^2}{2k(k-1)}.
\]
Every member of $N_{\mathrm J}(E(H_0))$ is a missing edge of $H$,
since an edge of $H$ adjacent to an edge of $H_0$ belongs to the
same tight component. Thus \cref{lem:Johnson} gives
\[
\beta
\ge1-\mu(H)
\ge\frac{|N_{\mathrm J}(E(H_0))|}{\binom nk}\ge\Phi_k(\mu(H_0))-\frac{\eps^2}{4k(k-1)}
>\beta,
\]
a contradiction. This completes the induction.
\end{proof}

\begin{proof}[Proof of \cref{cor:components}]
Fix an odd $r\ge3$ and $0<\eps<\beta_r$.
Choose $n_0$ sufficiently large that \cref{thm:components}
applies in uniformity $r$ with error $\eps/8$, and that
$\binom nr\ge8/\eps$ whenever $n\ge n_0$.
Let $H$ be an $r$-graph on $n\ge n_0$ vertices with
$\mu(H)\ge q_r+\eps$.
Suppose, for a contradiction, that every tight component of $H$
has fewer than $(1/2+\eps/4)\binom nr$ edges.

By \cref{thm:components}, $H$ has at least one tight
component with at least $(1/2-\eps/8)\binom nr$ edges.
Since distinct components have disjoint edge sets and
$1/2-\eps/8>1/3$, there are at most two such components.
From each such component, delete edges one at a time until fewer
than $(1/2-\eps/8)\binom nr$ remain, and let $H'$ be the resulting
spanning subgraph.
By our assumption, at most $(3\eps/8)\binom nr+1$ edges are
deleted from each component. Thus the total number of deleted
edges is at most
\[
2\left(\frac{3\eps}{8}\binom nr+1\right)
\le\eps\binom nr,
\]
and hence $\mu(H')\ge q_r$.

Deleting edges cannot merge tight components, so every tight
component of $H'$ has fewer than $(1/2-\eps/8)\binom nr$ edges.
This contradicts \cref{thm:components}, applied to $H'$
with error $\eps/8$.
\end{proof}

\section{Proof of \texorpdfstring{\cref{thm:main_all_odd}}{the Hamilton threshold theorem}}
\label{sec:hamilton}

In this section, we prove \cref{thm:main_all_odd} using the Hamilton
framework of Lang and
Sanhueza-Matamala~\cite{LangSanhuezaBandwidth}.
We first recall the framework and explain how to construct one
from tight components in the links.

\subsection{Hamilton frameworks}
We begin with the Hamilton-framework formalism of Lang and
Sanhueza-Matamala~\cite{LangSanhuezaBandwidth}, in the specialized form used
throughout the paper.
A \emph{fractional matching} in a $k$-graph $H$ is a function
$\omega:E(H)\to[0,1]$ such that
$\sum_{e\ni v}\omega(e)\le1$ for every $v\in V(H)$.
Its size is $\sum_{e\in E(H)}\omega(e)$, and it is \emph{perfect}
if its size is $|V(H)|/k$.
Write $\nu(H)$ and $\nu^*(H)$ for the maximum sizes of a matching
and a fractional matching in $H$, respectively.
We use the following definition for tight cycles.

\begin{definition}[{\cite{LangSanhuezaBandwidth}}]
\label{def:hamilton-framework}
Let $\cP_n$ be a family of $n$-vertex $k$-graphs.
We say that $\cP_n$ \emph{admits a Hamilton framework} if there exists a spanning subgraph $H^\star\subseteq H$ for each
$H\in\cP_n$ such that the following conditions hold.
\begin{enumerate}[label=\textup{(F\arabic*)},leftmargin=3em]
\item\label{item:F1} $H^\star$ is a tight component of $H$.
\item\label{item:F2} $H^\star$ has a perfect fractional matching.
\item\label{item:F3} $H^\star$ contains a closed tight walk of length $1$ modulo $k$.
\item\label{item:F4}
If $H_1=H^+-x$ and $H_2=H^+-y$ belong to $\cP_n$, where
$H^+$ is an $(n+1)$-vertex $k$-graph and $x\ne y$, then
$H_1^\star\cup H_2^\star$ is tightly connected.
\end{enumerate}
\end{definition}

For integers $1\le d<k$, the \emph{Hamilton framework threshold}
$\delta_d^{\HF}(k)$ is the infimum of all $\delta\in[0,1]$ such
that, for every $\eps>0$ and all sufficiently large $n$, the family
of all $n$-vertex $k$-graphs $H$ with
$\delta_d(H)\ge(\delta+\eps)\binom{n-d}{k-d}$ admits a Hamilton
framework.
The bandwidth theorem of Lang and
Sanhueza-Matamala~\cite[Theorem~1.1]{LangSanhuezaBandwidth},
with all clusters of size one, gives the following
consequence.

\begin{lemma}[{\cite{LangSanhuezaBandwidth}}]\label{thm:HC_HF}
For every $1\le d<k$, we have
$\delta_d^{\HC}(k)\le\delta_d^{\HF}(k)$.
\end{lemma}

For odd $r\ge3$ and $k\ge s_r$, let $d=k-r$.
By \cref{thm:HC_HF}, to prove
\cref{thm:main_all_odd} it suffices to show that
$\delta_d^{\HF}(k)\le q_r$.
Thus, for each fixed $\eps>0$ and all sufficiently large $n$, we
must construct a Hamilton framework for the family of $n$-vertex
$k$-graphs with minimum $d$-degree at least
$(q_r+\eps)\binom{n-d}r$.
We will select the large tight component in each $d$-link using
\cref{cor:components}. We first state the results used to verify the framework conditions.

\subsection{Tools for constructing the framework}

For $1\le d\le k-2$, a \emph{$d$-vicinity} in a $k$-graph $H$ is
a collection $\V=\{H_S\}_{S\in\binom{V(H)}d}$ with
$H_S\subseteq L_H(S)$.
It generates the spanning $k$-graph $H^\circ$ with
\[
E(H^\circ)=
\{S\cup e:S\in\tbinom{V(H)}d,\ e\in E(H_S)\}.
\]

The following two configurations are used to obtain tight
connectivity and a closed tight walk of the required length.
A \emph{switcher} in an $r$-graph $H$ is an edge $e$ with a vertex
$a\in e$ such that, for every $b\in e$, there is a vertex
$z\in V(H)$ with $(e\setminus\{a\})\cup\{z\}, (e\setminus\{b\})\cup\{z\}\in E(H)$.
An \emph{arc} of a $d$-vicinity $\V$ is an ordered $(k+1)$-tuple
$(v_1,\ldots,v_{k+1})$ of distinct vertices such that
\[
\{v_{d+1},\ldots,v_k\}\in E(H_{\{v_1,\ldots,v_d\}})~~\text{and}
~~\{v_{d+2},\ldots,v_{k+1}\}\in E(H_{\{v_2,\ldots,v_{d+1}\}}).
\]
We use the following results of Lang and Sanhueza-Matamala~\cite[Lemmas~4.19, 4.21 and~4.24]{LangSanhuezaBandwidth}.

\begin{lemma}[{\cite{LangSanhuezaBandwidth}}]
\label{lem:vicinity-toolkit}
Let $1\le d\le k-2$ and $r=k-d$.
\begin{enumerate}[label=\textup{(\roman*)},leftmargin=2.7em]
\item\label{item:connectivity-tool}
Suppose that a $d$-vicinity $\V$ in a $k$-graph has tightly
connected members, every two members share an $r$-edge, every
member contains a switcher, and $\V$ has an arc.
Then the generated $k$-graph is tightly connected and contains a
closed tight walk of length $1$ modulo $k$.
\item\label{item:arc-tool}
Let $\delta_0\in(0,1)$ satisfy
$\delta_0+\delta_0^{1-1/r}>1$.
For all sufficiently large $n$, every $d$-vicinity in an
$n$-vertex $k$-graph whose members each have at least
$\delta_0\binom{n-d}r$ edges has an arc.
\item\label{item:switcher-tool}
For every $\zeta>0$ and all sufficiently large $n$, the following
holds. Let $H$ be an $r$-graph on $n$ vertices and let
$H_0\subseteq H$ be a nonempty spanning subgraph. If
\[
\mu(H)\ge\frac r{r+2\sqrt{r-1}}+\zeta
\quad\text{and}\quad
\frac{\mu(H_0)}{\mu(\bd H_0)}
\ge\frac r{r+2\sqrt{r-1}}+\zeta,
\]
then $H_0$ contains a switcher.
\end{enumerate}
\end{lemma}

For the fractional matching condition, we use the following
lemma. It follows from the proof of
\cite[Proposition~1.1]{AlonEtAl2012} and is stated in this form in
\cite[Lemma~4.20]{LangSanhuezaBandwidth}.

\begin{lemma}[\cite{AlonEtAl2012,LangSanhuezaBandwidth}]\label{lem:matching-transfer}
Let $1\le d\le k-2$ and $n\ge k$, and let $H$ be an $n$-vertex
$k$-graph. If $\nu^*(L_H(S))\ge n/k$ for every
$S\in\binom{V(H)}d$, then $H$ has a perfect fractional matching.
\end{lemma}

We also use Frankl's shadow--matching inequality~\cite[Theorem~1.2]{Frankl2013}.

\begin{theorem}[\cite{Frankl2013}]\label{thm:Frankl}
For every $r\ge2$, every $r$-graph $H$ satisfies
$e(H)\le\nu(H)|\bd H|$.
\end{theorem}

\subsection{Constructing the framework}

\begin{proof}[Proof of \cref{thm:main_all_odd}]
For $r=1$, we have $s_1=2$ and $q_1=1/2$.
The case $k=2$ is Dirac's theorem~\cite{Dirac52}, and the cases
$k\ge3$ follow from the codegree result of R\"odl, Ruci\'nski,
and Szemer\'edi~\cite{RodlRucinskiSzemeredi2008}.
Suppose now that $r\ge3$ is odd, let $k\ge s_r$, and set
$d=k-r\ge1$.
Fix $0<\eps<\beta_r$, and let $n$ be sufficiently large in terms
of $k,r,\eps$.
Let $\cP_n$ be the family of all $n$-vertex $k$-graphs $H$ with
$\delta_d(H)\ge(q_r+\eps)\binom{n-d}r$. By \cref{def:hamilton-framework} and~\cref{thm:HC_HF}, it suffices to construct a Hamilton framework for $\mathcal P_n$.

For each $H\in\cP_n$ and $S\in\binom{V(H)}d$, we have
$\mu(L_H(S))\ge q_r+\eps$.
By \cref{cor:components}, the link $L_H(S)$ has a tight component
$H_S$ with at least $(1/2+\eps/4)\binom{n-d}r$ edges.
This component is unique, since distinct components have disjoint
edge sets.
The selected components form a $d$-vicinity; let $H^\circ$ be its
generated $k$-graph. 

\medskip\noindent\textbf{Verifying~\ref{item:F2}.}
Fix $S\in\binom{V(H)}d$, and let $H'_S$ be the spanning subgraph
of $L_H(S)$ with edge set $E(L_H(S))\setminus E(H_S)$.
Counting pairs $(T,e)$ with $e\in E(H'_S)$ and
$T\in\binom e{r-1}$ gives
$r e(H'_S)\le(n-k+1)|\bd H'_S|$.
Thus $\mu(\bd H'_S)\ge\mu(H'_S)=\mu(L_H(S))-\mu(H_S)$.
Since distinct tight components have disjoint shadows,
\[
\mu(\bd H_S)\le1-\mu(\bd H'_S)
\le1-\mu(L_H(S))+\mu(H_S).
\]
Using $\mu(H_S)\ge1/2$ and $\mu(L_H(S))\le1$, we obtain
\begin{equation}\label{eq:ratioeps}
\frac{\mu(H_S)}{\mu(\bd H_S)}
\ge\frac{\mu(H_S)}{1-\mu(L_H(S))+\mu(H_S)}\ge\frac1{3-2\mu(L_H(S))}
\ge\frac1{1+2\beta_r-2\eps},
\end{equation}
where the last inequality follows from $\mu(L_H(S))\ge q_r+\eps$.
By \cref{thm:Frankl},
\[
\nu^*(H_S)\ge\nu(H_S)
\ge\frac{e(H_S)}{|\bd H_S|}
 =\frac{n-k+1}{r}\,
   \frac{\mu(H_S)}{\mu(\bd H_S)}\ge\frac{n-k+1}{r(1+2\beta_r-2\eps)}>\frac nk,
\]
where the last inequality holds for sufficiently large $n$, since
$k-r(1+2\beta_r-2\eps)\ge2r\eps>0$ by $k\ge s_r$.
Since $H_S\subseteq L_{H^\circ}(S)$ for every $S$,
\cref{lem:matching-transfer} gives a perfect fractional matching
in $H^\circ$.
In particular, every vertex of $H$ lies in an edge of $H^\circ$.

\medskip\noindent\textbf{Verifying~\ref{item:F1} and~\ref{item:F3}.}
For any $S,T\in\binom{V(H)}d$, both $E(H_S)$ and $E(H_T)$
are subsets of $\binom{V(H)}r$. For sufficiently large $n$,
\[
e(H_S)+e(H_T)
\ge\left(1+\frac{\eps}{2}\right)\binom{n-d}r
>\binom nr.
\]
Thus, $H_S$ and $H_T$ share an $r$-edge, which is necessarily
disjoint from $S\cup T$.

We next check the constants required to find switchers.
Since $(2j-1)(2j+1)<(2j)^2$, the definition of $\beta_r$ gives
\[
\beta_r^2
=\prod_{j=1}^{(r+1)/2}\frac{(2j-1)^2}{(2j)^2}
<\prod_{j=1}^{(r+1)/2}\frac{2j-1}{2j+1}
=\frac1{r+2}.
\]
Therefore, $r^2\beta_r^2<r^2/(r+2)<r-1$ for $r\ge3$, and
\[
q_r=1-\beta_r>\frac{r-\sqrt{r-1}}{r}>\frac r{r+2\sqrt{r-1}}.
\]

For every $S$, we have $\mu(L_H(S))\ge q_r+\eps>
\frac r{r+2\sqrt{r-1}}$.
Moreover, by \eqref{eq:ratioeps} and $r\beta_r<\sqrt{r-1}$,
\[
\frac{\mu(H_S)}{\mu(\bd H_S)}
\ge\frac1{1+2\beta_r-2\eps}
>\frac1{1+2\beta_r}
>\frac r{r+2\sqrt{r-1}}.
\]
Thus, \cref{lem:vicinity-toolkit}\textup{(iii)}, applied to
$L_H(S)$ and $H_S$, gives a switcher in every $H_S$.

Since $\mu(H_S)\ge1/2+\eps/4$ for every $S$ and
$1/2+(1/2)^{1-1/r}>1$, \cref{lem:vicinity-toolkit}\textup{(ii)}, applied with
$\delta_0=1/2$, gives an arc.
Each $H_S$ is tightly connected, and any two share an edge,
as shown above. Hence, by~\cref{lem:vicinity-toolkit}\textup{(i)}, $H^\circ$
is tightly connected and contains a closed tight walk of length
$1$ modulo $k$.
Let $H^\star$ be the tight component of $H$ containing $H^\circ$.
Since $H^\circ$ is spanning,
$H^\star$ satisfies
\textup{(F1)--(F3)}.

\medskip\noindent\textbf{Verifying~\ref{item:F4}.}
Suppose that $H_1=H^+-x$ and $H_2=H^+-y$ belong to $\cP_n$,
where $H^+$ has $n+1$ vertices and $x\ne y$.
Choose a $d$-set $S\subseteq V(H^+)\setminus\{x,y\}$.
For $i\in\{1,2\}$, let $H_{i,S}$ be the component selected in
$L_{H_i}(S)$.
Each $H_{i,S}$ has at least $(1/2+\eps/4)\binom{n-d}r$ edges.
For sufficiently large $n$, $H_{1,S}$ and $H_{2,S}$ share an
edge $e$, since $ e(H_{1,S})+e(H_{2,S})\ge\left(1+\frac{\varepsilon}{2}\right)\binom{n-d}{r} >\binom{n+1-d}{r}.$
Then $S\cup e$ belongs to both $H_1^\star$ and $H_2^\star$.
Their union is therefore tightly connected.

Thus, $\cP_n$ admits a Hamilton framework.
Since this holds for every fixed $0<\eps<\beta_r$ and all
sufficiently large $n$, we have $\delta_d^{\HF}(k)\le q_r$.
Together with \cref{thm:HC_HF} and \eqref{eq:lowerbound}, this proves
\cref{thm:main_all_odd}.
\end{proof}

\section{Concluding remarks}\label{sec:conclusion}

In this paper, we established an edge-density condition for a large
tight component in every uniformity. More precisely,
\cref{thm:components} shows that every sufficiently large
$k$-graph of density at least $\widehat q_k$ has a tight component
of density $1/2-o(1)$. 
In odd uniformity, \cref{cor:components} gives a quantitative
strengthening above the critical density. As an application, \cref{thm:main_all_odd}
establishes the conjectured tight Hamilton cycle
threshold~\cite[Conjecture~2.3]{LangSanhuezaBandwidth} for odd $r$
and $k\ge s_r$, showing that the Han--Zhao lower bound is
asymptotically sharp in this range.

For odd $r\ge5$, a natural next step is to extend the equality
$\delta_{k-r}^{\HC}(k)=q_r$ to all $k\ge r+1$.
The restriction $k\ge s_r$ enters only through the
fractional-matching estimate in the selected link components,
so stronger matching bounds would extend the range of our
argument.  Another question is to determine the optimal density
forcing a tight component of density $1/2-o(1)$ in even uniformities
$k\ge4$. The sufficient bound $q_{k+1}$ in
\cref{thm:components} comes from passing to the next odd
uniformity in the neighborhood estimate. A direct analysis of
$\Phi_k$ may improve this bound, although the even-link Hamilton
problem also requires the compatibility and matching properties
discussed in the introduction.

\section*{Acknowledgments and statement on AI use}

We thank Fan Chang for bringing to our attention the discrete isoperimetric inequality for the Hamming graph with vertex set
$[n]^k$; our initial approach used the corresponding continuous
version on $[0,1]^k$.

The authors used ChatGPT during the early exploratory stage of this project. More precisely, in May 2026, GPT-5.5 directed us to Harper's isoperimetric theorem as a potentially relevant tool for estimating neighborhood sizes in Hamming graphs and assisted with preliminary calculations for $r=5,7$.
The authors subsequently studied the theorem, which enabled us to explore some initial cases of the problem and motivated our use of isoperimetric methods.

The main mathematical development of the present paper was subsequently carried out by the authors. In particular, the authors developed the general proof for all odd uniformities, including the analysis of the neighborhood bound and the induction on the uniformity via vertex links.
All mathematical arguments and proofs in the manuscript were written and checked by the authors.
 
\bibliographystyle{abbrv}
\bibliography{ref}

\end{document}